\documentclass[11pt, reqno, letterpaper]{amsart}

\usepackage{esint}
\usepackage{amssymb}
\usepackage{verbatim}

\usepackage{bbm}
\usepackage[%pagebackref,
hypertexnames=false, colorlinks,
linkcolor=black,citecolor=black,urlcolor=black, linktocpage=true
]%
{hyperref} %,hypertexnames=false,colorlinks,[pagebackref]
\hypersetup{bookmarksdepth=3}

	\normalbaselines						  %

\usepackage{graphicx}
\usepackage{%stmaryrd,
mathrsfs}

\usepackage[T1]{fontenc}

\theoremstyle{definition}

\newtheorem*{df*}{Definition}

\theoremstyle{remark}

\newtheorem*{rem*}{Remark}

\numberwithin{equation}{section}

\newcommand{\ddf}{\buildrel\mathrm{def}\over=}

\newtheorem{theorem}{Theorem}

\newtheorem{lemma}{Lemma}
\newtheorem{corollary}{Corollary}

\newtheorem{definition}{Definition}

\newcommand{\R}{\mathbb{R}}

\newcommand{\eps}{\varepsilon}

\usepackage{graphicx}
\DeclareGraphicsExtensions{.eps}

\newcommand{\av}[2]{\langle {#1}\rangle_{{}_{#2}}}

\begin{document}

\title[The maximal function]{Lower bounds for uncentered maximal functions in any dimension}

%\author{Robert Bryant}
%\thanks{}
%\address{Department of Mathematics,  Duke University, NC}
%\email{}

\author{Paata Ivanisvili}
%\thanks{PI is partially supported by the Hausdorff Institute for Mathematics, Bonn, Germany}
\address{Department of Mathematics,  Kent State University, Kent, OH 44240, USA}
\email{ivanishvili.paata@gmail.com}

\author{Benjamin Jaye}
\thanks{B.J. is partially supported by the NSF grant DMS-1500881}
\address{Department of Mathematics,  Kent State University, Kent, OH 44240, USA}
\email{bjaye@kent.edu}

\author{Fedor Nazarov}
\thanks{F.N. supported in part by NSF DMS-1265623.}
\address{Department of Mathematics,  Kent State University, Kent, OH 44240, USA}
\email{nazarov@math.kent.edu}

\makeatletter
\@namedef{subjclassname@2010}{
  \textup{2010} Mathematics Subject Classification}
\makeatother


% 42B	Harmonic analysis in several variables
% 42B20	Singular and oscillatory integrals (Calder?on-Zygmund, etc.)
% 42B35	Function spaces arising in harmonic analysis

% 47A	General theory of linear operators
% 47A30	Norms (inequalities, more than one norm, etc.)

%{30E20, 47B37, 47B40, 30D55.}
%
% 30D55	$H^p$-classes (1980-2009)
% 30E20	Integration, integrals of Cauchy type, integral representations of analytic functions
%
% 47B   	Special classes of linear operators
% 47B37	Operators on special spaces (weighted shifts, operators on sequence spaces, etc.)
% 47B40	Spectral operators, decomposable operators, well-bounded operators, etc.


\begin{abstract}
In this paper we address the following question: given $ p\in (1,\infty)$, $n \geq 1$, does there exists a constant $A(p,n)>1$ such that $\| M f\|_{L^{p}}\geq A(n,p) \| f\|_{L^{p}}$ for any nonnegative $f \in L^{p}(\mathbb{R}^{n})$, where $Mf$ is a maximal function operator defined over the family of shifts and dilates of a centrally symmetric convex body.  The inequality fails in general for the centered maximal function operator, but nevertheless we give an affirmative answer to the question for the uncentered maximal function operator and the {\em almost centered} maximal function operator.  In addition, we also present the Bellman function approach of  Melas, Nikolidakis and Stavropoulos  to  maximal function operators defined over various types of families of sets, and in case of parallelepipeds we will show that   $A(n,p)=\left(\frac{p}{p-1}\right)^{1/p}$. 
\end{abstract}

\date{}
\maketitle

\setcounter{equation}{0}
\setcounter{theorem}{0}

\section{Maximal function operators and main results}
\subsection{Centrally symmetric convex bodies}
 Fix any centrally symmetric convex body $K$ in $\mathbb{R}^{n}$ (that is, a compact convex set with non-empty interior). Let $\mathcal{K}$ be the family of all shifts of dilations of $K$.  For $\lambda \in [0,1]$, and a  centrally symmetric convex body $S$, set $\lambda S$  to be the image of $S$ under the homothety  with the center of the $S$, and ratio $\lambda$. If $\lambda =0$ then $\lambda S \ddf \{ x\}$ where $x$ is the center of $S$. Given any nonnegative locally integrable function $f$, we define the maximal operator
\begin{align*}
(M_{\lambda}f)(x) \ddf \sup_{S \in \mathcal{K}:\, \lambda S \ni x} \frac{1}{|S|}\int_{S} f,
\end{align*}
where $|S|$ denotes Lebesgue measure of the set $S$. Notice that $M_{0}f$ is the usual centered maximal function operator, while $M_{1} f$ is the uncentered maximal function operator.  Our first  main result is the following theorem.
\begin{theorem}\label{mth}
Fix  $p\in (1,\infty)$, $n \geq 1$, and  $\lambda\in (0,1]$. There exists a constant $A(p,n,\lambda)>1$ such that
\begin{align}\label{maxin}
\| M_{\lambda} f\|_{L^{p}} \geq A(n,p,\lambda)\|f\|_{L^{p}}, \quad \text{for all} \quad f\geq 0, \; f \in L^{p}(\mathbb{R}^{n}).
\end{align}
\end{theorem}

This result answers a question raised to the authors by Andrei Lerner during his visit to Kent State. In the case $n=1$ and $\lambda =1$, Theorem~\ref{maxin} is due to Lerner \cite{ler}.  There had recently been some activity in understanding the analogous problem for dyadic maximal operators, see  A. Melas, E. Nikolidakis, Th. Stavropoulos \cite{MNS1} and Section 1.2 below, but Theorem \ref{mth} appears to be new in general for even the uncentered maximal operator if  $n>1$. 

Theorem~\ref{mth} does not hold in general for the centered maximal function operator $M_{0}f$. Indeed, let $K$ be the unit ball in $\mathbb{R}^{n}$, $n \geq 3$. Take $f(x) = \min\{|x|^{2-n},1\}$ where $p>\frac{n}{n-2}$. Then $f \in L^{p}(\mathbb{R}^{n})$, and, since $f$ is superharmonic, we have $M_{0} f = f$ and so $\|M_{0}f\|_{L^{p}} = \|f\|_{L^{p}}.$  

% (Commented) We do not have such examples for the remaining set of exponents $p$ and $n$  because if we had a function  $f \in L^{p}(\mathbb{R}^{n})$ with $\| M_{0} f \|_{p} = \| f\|_{p}$ then we would have $M_{0}f =f$, and thus $f$ would be superharominic.  

On the other hand if $n=1,2$ and $1<p<\infty$, or $n \geq 3$ and $p\leq \frac{n}{n-2}$, then any function $f\in L^p(\mathbb{R}^n)$ satisfying $\|M_0(f)\|_{L^p} = \|f\|_{L^p}$ (or equivalently $M_0f=f$) must be zero, see for instance~\cite{soul, martini, fiorenza}. Hence for this set of exponents we have $\|M_{0} f\|_{p} > \|f\|_{p}$ for each $f \in L^{p}$.  We should also mention~\cite{janakiraman} where (\ref{maxin}) was investigated in the case  $n=1$, $p=1$, and $\lambda=0$, with  $\| M_{0} f\|_{L^{1}}$ replaced by the weak (1,1) norm, and it was shown that in that case there is no such $A>1$. 

%However, we do not know if there exists $A>1$ for which $\| M_{0} f\|_{p} \geq A \|f\|_{p}$ for all $f \in L^{p}(\mathbb{R}^{n})$ in this case. 

In Section~\ref{unc} we provide a simple proof of Theorem~\ref{mth} in the case when $\lambda=1$. In Section~\ref{acen} we adapt the proof to the general case  $\lambda\in (0,1)$.

Let us now make a few remarks about the nature of the constant $A(n,p,\lambda)$.  Our proof yields   $A(n,p,\lambda) = 1+\eps(n,p,\lambda)$, where $\eps(n,p,\lambda)$ decays exponentially with the dimension $n$.  This dependence is a direct consequence of our use of the Besicovitch covering lemma, and we do not know what the true dependence should be.   The constant $\eps(n,p,\lambda)$ must (in general) tend to zero as $\lambda\rightarrow 0$, but the dependence on $\lambda$ in our argument is essentially qualitative, as we rely on a compactness argument (Lemma \ref{either} below).  Finally, $\eps(n,p,\lambda)$ is comparable with  $\tfrac{1}{p-1}$ as $p\rightarrow 1^+$, as it should be.

% we prove Theorem~\ref{mth} when $k=1$. In Section~\ref{acen} we give a complete proof of the Theorem in the general case $k \in (0,1)$. We also present good estimates for the constant $A(p,n,k)$.

\subsection{Other maximal functions}

We denote by  $\av{f}{A}$ the integral average of $f$ over a measurable set $A$, i.e., $\av{f}{A} = \frac{1}{|A|}\int_{A}f$. If $|A|=0$ then we set $\av{f}{A}=0$. Let $\mathcal{Q}$ be  some family of convex bodies in  $\mathbb{R}^{n}$. %with the property that $\emptyset \in \mathcal{Q}$.
We define the maximal function operator over the family $Q$ as follows:
\begin{align*}
M_{\mathcal{Q}}f(x) = \sup_{Q \ni x:\; Q \in \mathcal{Q}}  \av{f}{A}.
\end{align*}
\begin{definition}
We say that the family $\mathcal{Q}$ is $\lambda$-dense ($\lambda >1$) if for any locally integrable $f \geq 0$ and  any $Q \in \mathcal{Q}$ there exists  {\em a filtration} $\{ \mathcal{F}_{n}\}_{n=0}^{\infty}$  associated to $Q$ and $f$ such that
\begin{itemize}
\item[1.] $\mathcal{F}_{0} = \{ Q\}$.
\item[2.] For any $n \geq 0$, $\mathcal{F}_{n}$ consists by  at most countable number of sets  from $\mathcal{Q}$.
\item[3.] $Q=\cup_{P \in \mathcal{F}_{n}}P$ for any $n \geq 0$.
\item[4.] The elements of  $\mathcal{F}_{n}$ are {\em almost disjoint}, i.e., $|P\cap R|=0$ for any different $P, R \in \mathcal{F}_{n}$ and for any $n \geq 1$.
\item[5.] $\mathcal{F}_{n+1}$ is a refinement of  $\mathcal{F}_{n}$, i.e.,  for any $P \in \mathcal{F}_{n}$ there is a family of sets $\mathrm{ch}(P) \subset  \mathcal{F}_{n+1}$ such that $P=\cup_{R\in \mathrm{ch}(Q)} R$.
\item[6.] $\lim_{n \to \infty} \sup_{P \in \mathcal{F}_{n}}\mathrm{diam}(P)=0$.
\item[7.] $\sup_{R \in \mathrm{ch}(P)}\{ \av{f}{R}\}\leq \lambda \av{f}{P}$ for any $P \in \mathcal{F}_{n}$ and any $n \geq 0$.

\end{itemize}

\end{definition}

We will show  in Lemma~\ref{ldens} that the  set of all parallelepipeds with sides  parallel to some fixed linearly independent $n$ vectors in $\mathbb{R}^{n}$   is $\lambda$-dense for any $\lambda >1$. In particular the set of all intervals on the real line is $\lambda$-dense for any $\lambda>1$.

We say that the family $\mathcal{Q}$ is {\em exhaustive} if for any compact $E \subset \mathbb{R}^{n}$ there exists $Q \in \mathcal{Q}$ such that $E \subset Q$. Our second main result is the following theorem.
\begin{theorem}\label{smth}
If the family $Q$ is exhaustive and $\lambda$-dense for some $\lambda>1$, then
\begin{align*}
\| M_{\mathcal{Q}} f\|_{L^{p}}\geq \left(\frac{\lambda^{p}-1}{\lambda^{p}-\lambda}\right)^{1/p}\| f\|_{L^{p}} \quad \text{for all} \quad f\geq 0, \; f\in L^{p}(\mathbb{R}^{n}).
\end{align*}
\end{theorem}
Since $\frac{\lambda^{p}-1}{\lambda^{p}-\lambda}$ is decreasing in $\lambda$ for $\lambda >1$, we want to make $\lambda$ as close to $1$ as possible. Also notice that if the family $\mathcal{Q}$  is $\lambda$-dense for every $\lambda>1$ then we can take $\lim_{\lambda \to 1+}\left( \frac{\lambda-1}{\lambda^{p}-\lambda}\right)^{1/p}=\left(\frac{p}{p-1}\right)^{1/p}$. Thus  by Lemma~\ref{ldens} we obtain the following corollary.
\begin{corollary}\label{cor1}
Let $\mathcal{Q}$ be the set of all parallelepipeds with  sides parallel to some fixed linearly independent $n$ vectors in $\mathbb{R}^{n}$. Then
\begin{align}\label{sharp}
\| M_{\mathcal{Q}} f\|_{L^{p}}\geq \left(\frac{p}{p-1}\right)^{1/p} \| f\|_{L^{p}} \quad \text{for all} \quad f\geq 0, \; f\in L^{p}(\mathbb{R}^{n}).
\end{align}
\end{corollary}
We wonder if the constant $\bigl(\frac{p}{p-1}\bigl)^{1/p}$ is the best possible in Corollary~\ref{cor1}. In~\cite{ler} Lerner obtained (\ref{sharp}) in one dimensional case, i.e., $n=1$. We think that the argument presented in~\cite{ler} (see also Remark~\ref{sun} in Section~\ref{belmanchik}) does not extend to high dimensions $n>1$.

\bigskip

 It is not difficult to see that the set of dyadic cubes $\mathcal{Q}$   in $\mathbb{R}^{n}$ is $2^{n}$-dense, and this is the smallest $\lambda$ one can choose. In this case we recover the inequality proved by  Melas, Nikolidakis and Stavropoulos (see~\cite{MNS1}).
 \begin{corollary}
 Let $\mathcal{Q}$ be a set of dyadic cubes in $\mathbb{R}^{n}$. Then
  \begin{align*}
\| M_{\mathcal{Q}} f\|_{L^{p}}\geq \left(\frac{2^{np}-1}{2^{np}-2^{n}}\right)^{1/p} \| f\|_{L^{p}} \quad \text{for all} \quad f\geq 0, \; f\in L^{p}(\mathbb{R}^{n}).
\end{align*}
\end{corollary}

\section{Uncentered maximal function operator and convex bodies}\label{unc}
In this section we present a simple  proof of  Theorem~\ref{mth} in the case $k=1$.

Assume that $f \geq 0$ is continuous with compact support. Fix $t >0$ and consider the set
\begin{align*}
\mathcal{K}(t) = \left \{  K \in \mathcal{K} : \int_{K} f = t|K| \right\}.
\end{align*}
Clearly for each $t>0$ the set of the centers of $K$ from $\mathcal{K}(t)$ belong to a bounded subset of $\mathbb{R}^{n}$. We apply the Besicovich covering lemma  (see Lemma~\ref{bcl} in the Appendices) to extract a countable subfamily $K_{t,j}\in \mathcal{K}(t)$ so that the function
\begin{align*}
\psi(x,t) = \sum_{j}\chi_{K_{t,j}}(x)
\end{align*}
satisfies the following properties:

\begin{itemize}
\item[1)] For all $x \in \mathbb{R}^{n}$, $t>0,$ we have $\psi(x,t)\leq B(n)$ with some constant $B(n)$ depending only on the dimension $n$;
\item[2)] If $t>M_{1}f(x)$ then $\psi(x,t)=0$;
\item[3)] If $f(x) >t,$ then $\psi(x,t) \geq 1$;
\item[4)] For every $t >0$, we have $\int_{\mathbb{R}^{n}} t\psi(x,t)dx = \int_{\mathbb{R}^{n}}\psi(x,t)f(x)dx$;
\end{itemize}

The first property follows from Lemma~\ref{bcl}  where $B(n)$ is a Besicovich constant. This number is independent of $t$ and depends only on the dimension $n$.

For the second property, if $t > M_{1}f(x)$ then no $K_{t,j}$ contains $x$. Indeed, otherwise if some $K_{t,\ell}$ contains $x$ then $M_{1}f(x) \geq \av{f}{K_{t,\ell}} = t$ by the choice of the family $\mathcal{K}(t)$.

To verify the third property assume $f(x)>t$.  Let $K(x)$ be a shift  of $K$ centered at $x$. By the intermediate value theorem the continuous function $p(s) =\av{f}{sK(x)}$ attains value $t$ for some finite positive number $s^{*}$. Then by the choice of the family $K_{t,j}$ there exists $K_{t,\ell}$ which contains the center of $s^{*}K(x)$, i.e., $x$, and the property follows.

The fourth property follows  immediately from the fact that $\int_{\mathbb{R}^{n}}t  \chi_{K_{t,j}}=\int_{\mathbb{R}^{n}}\chi_{K_{t,j}}f(x) dx$ for all $j$ and $t>0$.

Now the last property, after multiplying both sides by $t^{p-2}$ and integrating with respect to $t$ over the ray $(0,\infty)$, yields the equality:
\begin{align}\label{gr}
\int_{\mathbb{R}^{n}}\int_{0}^{\infty}t^{p-1}\psi(x,t)dtdx = \int_{\mathbb{R}^{n}}\int_{0}^{\infty}t^{p-2} \psi(x,t) f(x) dtdx.
\end{align}

On the left hand side of (\ref{gr}) we can restrict the integration with respect to $t$ to $[0,M_{1}f(x)]$ (by property 2). We will estimate the right hand side from below by restricting the inner integration to the interval $[0,f(x)]$. The obtained inequality (together with the properties 2 and 3) justifies the following chain of inequalities:
\begin{align*}
&\frac{B(n)}{p}(\| M_{1} f\|_{L^{p}}^{p} - \| f\|_{L^{p}}^{p})\geq \int_{\mathbb{R}^{n}}\int_{0}^{M_{1}f(x)}t^{p-1}\psi(x,t)dtdx - \int_{\mathbb{R}^{n}}\int_{0}^{f(x)}t^{p-1}\psi(x,t)dtdx \geq\\
 &\int_{\mathbb{R}^{n}}\int_{0}^{f(x)}t^{p-1}\psi(x,t)\left(\frac{f(x)}{t} -1\right)dtdx \geq \int_{\mathbb{R}^{n}}\int_{0}^{f(x)}t^{p-1}\left(\frac{f(x)}{t}-1 \right)dtdx = \frac{\|f\|_{L^{p}}^{p} }{p(p-1)}.
\end{align*}
Thus we obtain (\ref{maxin}) with the constant $A(p,n,1)=\bigl(1+\tfrac{1}{(p-1)B(n)} \bigl)^{1/p}$.

\section{Almost centered maximal function operator and convex bodies}\label{acen}
In this section we work with the operator $M_{\lambda}$ for $\lambda \in (0,1)$. Assume that $f \geq 0$ is continuous with compact support. Fix $ t >0$,  and consider the family of sets
\begin{align*}
\mathcal{K}(t) = \left\{  K \in \mathcal{K}\, : \, \av{f}{K}= t \right\}
\end{align*}
%By maximal we mean that  if there are two sets $K_{1}, K_{2}  \in \mathcal{K}$ such that $K_{1} \subset K_{2}$  and $\av{f}{K_{j}}=t$ then we do not
%include  $K_{1}$ to  $\mathcal{K}(t)$.
as before. Once again, we use the Besicovitch covering lemma (Lemma~\ref{bcl} below) to extract the family $K_{t,j}$ so that the sets $K_{t,j}$ cover the centers of the sets  in $\mathcal{K}(t)$. Set $\psi(x,t)=\sum_{j} \chi_{K_{t,j}}(x)$.

In precisely the same manner as in Section \ref{unc}, we notice the following properties:
\begin{itemize}
\item[1)] For all $x \in \mathbb{R}^{n}$, $t>0,$ we have  $\psi(x,t) \leq B(n)$  with some constant $B(n)$ depending only on the dimension $n$;
\item[2)] If $f(x) > t,$ then $\psi(x,t) \geq 1$;
\end{itemize}

%The first property is consequence of Lemma A. For the second property, notice that if $f(x)>t$ then there is the maximal set  $K' \in \mathcal{K}$ by inclusion with center at $x$ such that  $\av{f}{K'}=t$. Clearly $K$ belongs to $\mathcal{K}(t)$. Since the extracted cover has a set $K_{t,\ell}$ such that $k K_{t, \ell}$ contains $x$ then the larger set $K_{t, \ell}$ also contains $x$.

Unfortunately it is not true that if $t> M_{\lambda}f(x)$ then $\psi(x,t)=0$, and therefore we cannot repeat the proof as in the previous section.  However, to compensate for the lack  of this property we  will prove the following dichotomy.
\begin{lemma}\label{either}
For every $\varepsilon \in (0,1)$  there exists $\eta>0$ such that for every $K \in \mathcal{K}$ and any function $f\geq 0$,  either
\begin{align}
&\int_{K}M_{\lambda}f(x)dx \geq  (1+\eta) \int_{K} f(x)dx, \quad \text{or} \label{pmd}\\
&M_{\lambda} f(x)  \geq (1-\eps) \frac{1}{|K|}\int_{K} f(x)dx \quad \text{on} \quad (1-\varepsilon)K. \label{mmd}
\end{align}

\end{lemma}
Before we proceed to the proof of the lemma let us show how it implies the desired estimate.  For each $t>0$, the family  $\{  K_{t, j}\}$ can be divided into two subfamilies $\{ K'_{t,j}\}$ and $\{ K''_{t,j}\}$ so that  the sets $K'_{t,j}$ satisfy (\ref{pmd}), and the sets $K''_{t,j}$ satisfy (\ref{mmd}). Set
\begin{align*}
\psi_{1}(x,t) = \sum_{j} \chi_{K'_{t,j}}(x),  \quad  \quad \psi_{2}(x,t) = \sum_{j} \chi_{K''_{t,j}}(x) \quad \text{and} \quad \psi^{\eps}_{2}(x,t) = \sum_{j} \chi_{(1-\eps)K''_{t,j}}(x).
\end{align*}

Clearly $\psi_{1} +\psi_{2} \geq 1$ if $f(x)>t$, and $\psi_{1}, \psi_{2} \leq B(n)$ for all $t >0$ and $x \in \mathbb{R}^{n}$. We notice that
\begin{align}\label{vspomnil}
\text{if} \quad t > \lambda^{-n} M_{\lambda} f(x) \quad \text{then} \quad \psi_{1}(x,t)=0.
\end{align}
Indeed, otherwise there exists a set $K'_{t,\ell} \in \{ K'_{t,j}\}$ containing $x$ such that
\begin{align*}
t = \frac{1}{|K'_{t,\ell}|} \int_{K'_{t,\ell}} f(x)dx\leq  \frac{\lambda^{-n}}{|\lambda^{-1} K'_{t,\ell}|}\int_{\lambda^{-1}K'_{t,\ell}} \leq \lambda^{-n} M_{\lambda}f(x).
\end{align*}
Note that (\ref{pmd}) implies the following inequality
\begin{align}\label{krivoi}
\int_{\mathbb{R}^{n}}\int_{0}^{\infty} (M_{\lambda}f(x)  - f(x)) \psi_{1}(x,t) t^{p-2}dt dx \geq \eta  \int_{\mathbb{R}^{n}} \int_{0}^{\infty} f(x) \psi_{1}(x,t) t^{p-2}dtdx.
\end{align}

Since $\psi_1\leq B(n)$, and $M_{\lambda} f \geq f$, we have that
\begin{equation}\begin{split}\nonumber\frac{B(n)}{(p-1)\lambda^{-n(p-1)}}(\|M_{\lambda}f\|_{p}^{p} - \|f\|_{p}^{p})&\geq B(n) \int_{\mathbb{R}^{n}} [M_{\lambda}f(x)  - f(x)]\frac{(M_{\lambda}f)^{p-1}}{(p-1)\lambda^{-n(p-1)}} dx\\&\geq \int_{\mathbb{R}^{n}}\int_{0}^{\lambda^{-n}M_{\lambda}f(x)} [M_{\lambda}f(x)  - f(x)] \psi_{1}(x,t) t^{p-2}dt dx.
\end{split}\end{equation}
Notice that  (\ref{vspomnil}) enables us to extend the integration over $t$ in the inner integral on the right hand side, and so we derive from  (\ref{krivoi}) that
\begin{align}
\frac{B(n)}{\eta(p-1)\lambda^{-n(p-1)}}(\|M_{\lambda}f\|_{p}^{p} - \|f\|_{p}^{p})&\geq \frac{1}{\eta}\int_{\mathbb{R}^{n}}\int_{0}^{\infty} [M_{\lambda}f(x)  - f(x)] \psi_{1}(x,t) t^{p-2}dt dx \nonumber\\
&\geq  \int_{\mathbb{R}^{n}} \int_{0}^{\infty} f(x) \psi_{1}(x,t) t^{p-2}dtdx. \label{pirbolo}
\end{align}
Unfortunately we do know that $\psi_{1} \geq 1$ on the set $\{t<f(x)\}$, but instead only that $\psi_{1} +\psi_{2} \geq 1$, and so we need to invoke the function $\psi_{2}$.

We have $\int_{K''_{t,\ell}} t dx= \int_{K''_{t,\ell}}f(x)dx$, and so $(1-\varepsilon)^{-n} \int_{(1-\varepsilon)K''_{t,\ell}} tdx = \int_{K''_{t,\ell}}f(x)dx$.   Tonelli's theorem  therefore yields that
\begin{align}\label{meore}
(1-\varepsilon)^{-n} \int_{\mathbb{R}^{n}} \int_{0}^{\infty} \psi^{\eps}_{2}(x,t) t^{p-1} dtdx  =   \int_{\mathbb{R}^{n}}\int_{0}^{\infty} f(x)\psi_{2}(x,t)t^{p-2}dtdx.
\end{align}
Since for each $x \in (1-\varepsilon)K''_{t,j}$ we have $\frac{M_{\lambda}f(x)}{1-\varepsilon}>t$ then (\ref{meore}) can be rewritten as follows
\begin{align}\label{ragaca}
(1-\varepsilon)^{-n}\int_{\mathbb{R}^{n}} \int_{0}^{\frac{M_{\lambda} f(x)}{1-\varepsilon}}  \psi^{\eps}_{2}(x,t) t^{p-1} dtdx  =  \int_{\mathbb{R}^{n}}\int_{0}^{\infty} f(x)\psi_{2}(x,t)t^{p-2}dtdx.
\end{align}
In the left hand side of (\ref{ragaca}) we estimate $\psi^{\eps}_{2}$ from above by $  \psi_{1}+\psi_{2}$.  After combining the resulting inequality together with (\ref{pirbolo}) we obtain
\begin{align*}
&\frac{B(n)}{\eta(p-1)\lambda^{-n(p-1)}}(\|M_{\lambda} f\|^{p}_{p}-\|f\|_{p}^{p})+ (1-\varepsilon)^{-n} \int_{\mathbb{R}^{n}}\int_{0}^{\frac{M_{\lambda}f(x)}{1-\varepsilon}} t^{p-1}(\psi_{2}+\psi_{1})  \\
&\geq \int_{\mathbb{R}^{n}}\int_{0}^{\infty}f(x)(\psi_{1}+\psi_{2})t^{p-2}dtdx  \geq \int_{\mathbb{R}^{n}}\int_{0}^{f(x)}f(x)(\psi_{1}+\psi_{2})t^{p-2}dtdx.
\end{align*}
We now subtract the term $\int_{\R^n}\int_0^{f(x)}t^{p-1}(\psi_1+\psi_2)dtdx$ from  both sides of  previous inequality, which yields
\begin{align*}
&\frac{B(n)}{\eta (p-1)\lambda^{-n(p-1)}}(\|M_{\lambda} f\|^{p}_{p}-\|f\|_{p}^{p})+ (1-\varepsilon)^{-n}\int_{\mathbb{R}^{n}}\int_{f(x)}^{\frac{M_{\lambda}f(x)}{1-\varepsilon}} t^{p-1}(\psi_{2}+\psi_{1}) dtdx\\
&+((1-\eps)^{-n}-1)\int_{\mathbb{R}^{n}}\int_{0}^{f(x)}t^{p-1}(\psi_{1}+\psi_{2})dtds\geq  \int_{\mathbb{R}^{n}}\int_{0}^{f(x)}\left(\frac{f(x)}{t}-1\right) (\psi_{1}+\psi_{2})t^{p-1}dtdx.
\end{align*}
As in the previous section, we estimate the left hand side of the previous display from above using the inequality  $B(n) \geq \psi_{1}+\psi_{2}$, and the right hand side from below using that estimate $\psi_{1}+\psi_{2}\geq 1$ in the domain of integration (notice that $t<f(x)$). Finally we obtain
\begin{align*}
\frac{B(n)}{\eta (p-1)\lambda^{-n(p-1)}}&(\|M_{\lambda} f\|^{p}_{p}-\|f\|_{p}^{p})+ \frac{B(n) (1-\varepsilon)^{-n}}{p}\left(  \frac{\| M_{\lambda} f \|_{p}^{p}}{(1-\varepsilon)^{p}} - \|f\|_{p}^{p}\right) \\
&+\frac{((1-\eps)^{-n}-1)B(n)}{p} \| f\|_{L^{p}}^{p}  \geq \frac{\| f\|_{p}^{p}}{p(p-1)},
\end{align*}
which, when rearranged, becomes
\begin{align*}
\| M_{\lambda} f\|_{p}^{p} \geq \left(1+\frac{1-(1-\varepsilon)^{-n-p}+\frac{1}{B(n) (p-1)}}{\frac{p}{\eta (p-1) \lambda^{-n(p-1)}} +(1-\varepsilon)^{-n-p}} \right)\| f\|^{p}_{p}.
\end{align*}
Choosing $\varepsilon>0$ sufficiently small so that $1-(1-\varepsilon)^{-n-p}+\frac{1}{B(n) (p-1)} >0$   we obtain the desired estimate.

It remains to prove Lemma~\ref{either}.

\bigskip

{\em Proof of Lemma~\ref{either}.}
Since the maximal function $M_{\lambda}f(x)$ commutes with dilations and shifts,  i.e., $(M_{\lambda} f(\alpha \cdot +\beta))(x)=(M_{\lambda}f(\cdot))(\alpha x+\beta)$, it is enough to prove the lemma for some fixed $K_0 \in \mathcal{K}$ with $|K_0|=1$. Assume to the contrary that there exists $\varepsilon_{0}\in (0,1)$ and a sequence of non-negative functions $f_j$ that satisfy
$\int_{K_0}M_{\lambda}f_{j}(x)dx < (1+\tfrac{1}{j}) \int_{K_0}f_{j} dx$ while, for every $j$, the inequality $M_{\lambda} f_{j} \geq (1-\varepsilon_{0}) \frac{1}{|K_0|}\int_{K_0} f_{j} dx$ does not hold on $(1-\varepsilon_{0})K_0$.  By considering $\frac{\chi_{K_0}f_{j}}{\int_{K_0} f_{j}}$ we can assume that $\int_{K_0} f_{j}=1$ and $f_j$ is supported in $K_0$.

With a view to passing to a limit, we first claim that the sequence $f_j$ is uniformly integrable on $K_0$.

To this end, recall Stein's inequality \cite{S}, which states that there is a constant $C>0$ depending only on $n$, such that if $f$ is a non-negative function with $\int_{K_0} f=1$, then
$$\int_{K_0} f\ln(\max\{ 1, f\})dx\leq C\int_{K_0} M_{\lambda}(f) dx.
$$
For the benefit of the reader we include a proof in an appendix (see Lemma \ref{stein}).

Returning to our sequence $f_j$, we find that  $\int_{K_0}f_{j} \ln(\max\{1, f_j\})dx \leq C(1+\frac{1}{j})\leq 2C$ and this readily yields that the sequence $\{ f_{j}\}$ is uniformly integrable. 
% Indeed, to see this let $\eps>0$, $\delta>0$, and suppose that $\int_{E}f_j dx\geq \eps$ for some $j$, while $|E|<\delta$.  For $N>0$, we have that $\int_{E\cap \{f_j\geq N\}} f_j dx\geq \tfrac{\eps}{2}$ as long as $N\delta<\eps/2$.  But then
%$$2C\geq \int f_j \log(e+f_j)\geq \int_{E\cap \{f_j\geq N\}} f_j \log(e+f_j)dx \geq \frac{\eps}{2}\log (N).
%$$
%The right hand side can be chosen greater than $2C$ (which is absurd) if $N> e^{4C/\eps}$.  This choice is possible if $\delta <\tfrac{\eps}{2}e^{-4C/\eps}$, and uniform integrability follows: 
%$$ \int_{E} f_j dx<\eps \text{ for every }j\text{ whenever } |E|< \frac{\eps}{2}e^{-4C/\eps}.
%$$

Consequently, with the aid of the Dunford--Pettis theorem, we may (by passing to a subsequence if necessary)  find $f \in L^{1}(K_0)$ such that $f_{j} \to f$ in the $\sigma(L^{1}(K_0), L^{\infty}(K_0))$ topology (i.e. the sequence $f_j$ converges weakly to $f$ over bounded functions).

It is clear that $\int_{K_0} f=1$. Further we have that $\liminf_j M_{\lambda} f_{j} \geq M_{\lambda} f$ a.e. on $K_0$, and so Fatou's lemma yields that $\int_{K_0} M_{\lambda}f dx \leq 1$.

%   Since $\int_{K} \tilde{M}_{k} f_{j}(x) dx \leq (1+1/j)\int_{K}f_{j}$ and the sequence $\{ f_{j}\}$ is uniformly integrable we obtain that $\{ \tilde{M}_{k} f_{j}\}$ is uniformly integrable. Therefore $ \tilde{M}_{k} f_{j}   \to \tilde{M}_{k} f $ in $L^{1}(K)$ and thus $\int_{K} \tilde{M}_{k} f \leq 1$.

The properties $\int_{K_0} f=1$ and $\int_{K_0} M_{\lambda} f \leq 1$ imply that $M_{\lambda} f = f$ almost everywhere on $K_0$.  We will show that $f=1$ almost everywhere on $K_0$ and that this will contradict  our assumption that the inequality $M_{\lambda} f_{j} \geq (1-\varepsilon_{0})$ fails to hold  on $(1-\varepsilon_{0})K_0$ for sufficiently large $j$.

Fix $r>0$. Set $f_{r} = f * \varphi_{r}$,  where $\varphi\geq 0$  is a smooth bump function supported on $B(0,1)$ such that $\int \varphi =1$, and $\varphi_{r}(x) =r^{-n}\varphi(\frac{x}{r})$.  If $f$ is non-constant (a.e.) on $K_0$, then we can find for arbitrarily small $r>0$ a point $x_0\in K_r=\{x\in K_0: \text{dist}(x,\partial K_0)>r\}$ so that $\nabla f_r(x_0)\neq 0$.
 
 Notice that $M_{\lambda} f_{r} \leq \varphi_{r} * M_{\lambda} f = \varphi_{r}*f = f_{r}$ on $K_r$, and therefore $M_{\lambda} f_{r} = f_{r}$ on $K_r$. 
 
 Take any set $K' \in \mathcal{K}$ centered at $x_{0}$.  Then (for instance by expanding $f_{r}$ in a Taylor series), we see that there is a constant $C>0$, that may depend on $r$, such that  $|\av{f_{r}}{sK'}-f_{r}(x_{0})| \leq C \| D^{2} f_{r}\|_{L^{\infty}(K_0)} s^{2}$ for all sufficiently small $s$.  But then provided that $s$ is small enough to ensure that $\lambda sK'\subset K_r$, we have that  $f_{r}(z)=M_{\lambda}f_r(z) \geq \av{f}{sK'}$ for all $z \in \lambda s K'$.  Therefore, for sufficiently small $s$, we have 
\begin{align}\label{smoothedlower}
f_{r}(z) \geq \av{f_{r}}{sK'} \geq f_{r}(x_{0})-C\| D^{2} f_{r}\|_{L^{\infty}(K_0)} s^{2}\text{ for every }z \in s\lambda K'.
\end{align}
This inequality already contradicts to our assumption that the gradient is not zero at point $x_{0}$.  We therefore conclude that $f=1$ on $K_0$. 
%On the other hand, let $e$ be a unit vector such that $\langle\nabla f_r(x_{0}), e\rangle = -|\nabla f_{r}(x_{0})|$.  Let us choose $z\in \lambda sK'$ such that $z=x_0+\lambda s'e$, where $\tfrac{s}{C'}\leq s'\leq C's$ for some constant $C'$ depending only on the eccentricity of the family $\mathcal{K}$.   Then, provided that $s$ is small enough, we have that $f_r(z) \leq f_r(x_{0}) - c s|\nabla f_r(x_{0})|$ where $c>0$ is a constant depending on $r$, $\lambda$, and the eccentricity of the family $\mathcal{K}$. Combining this inequality with (\ref{smoothedlower}), we deduce that as long as $s$ is small enough, $$f_r(x_{0}) - c s|\nabla f_r(x_{0})| \geq f_{r}(x_{0})-C\| D^{2} f_{r}\|_{L^{\infty}(K_0)} s^{2},$$ which is absurd.  

Consider $(1-\varepsilon_0)K_0$. There are a finite number of sets $K_{\ell} \subset K_0$ with $K_{\ell} \in \mathcal{K}$ such that the sets $\lambda K_{\ell}$ cover $(1-\varepsilon_0)K_0$. Since the sequence $f_j$ converges weakly to the constant function $1$ over $L^{\infty}(K_0)$, we have that $\av{f_{j}}{K_{\ell}} \to 1$ as $j \to \infty$ for each $\ell$.  Consequently, we can choose sufficiently large $j_{0}>0$  such that $\av{f_{j}}{K_{\ell}}\geq 1-\eps_0$ for every $\ell$ and every $j\geq j_0$.  Thus $M_{\lambda} f_{j} \geq (1-\varepsilon_{0})$  on $(1-\varepsilon)K_0$ for all $j\geq j_{0}$.   This final contradiction completes the proof of the lemma.  $\hfill \square$

\section{Bellman function approach}\label{belmanchik}
\subsection{The proof of Theorem~\ref{smth}.}
In this section we will prove Theorem~\ref{smth}. Given compactly supported continuous  $f\geq 0$, for any $Q \in \mathcal{Q}$ we set
\begin{align*}
(x_{Q}(f), y_{Q}(f), z_{Q}(f)) \ddf (\av{f}{Q}, \av{f^{p}}{Q}, \sup_{R \in \mathcal{Q} : \, R \supseteq Q}\av{f}{R}).
\end{align*}
Sometimes we will omit the variable $f$ and we just write $(x_{Q}, y_{Q}, z_{Q})$.
For a real number $a$ we set $(a)_{+} \ddf  \max\{a,0\}$. First we will  prove the following lemma.
\begin{lemma}\label{ml}
If the family $\mathcal{Q}$ is $\lambda$-dense, then for any $S \in \mathcal{Q}$ and any $p>1$ we have
\begin{align*}
\av{(M_{\mathcal{Q}}f)^{p}}{S}\geq z_{S}^{p}(f) + \frac{\lambda^{p}-1}{\lambda^{p}-\lambda}(y_{S}(f)-x_{S}(f)z_{S}^{p-1}(f))_{+}.
\end{align*}
\end{lemma}
First let us explain that the lemma implies Theorem~\ref{smth}. Indeed, since $\mathcal{Q}$ is exhaustive we can find a sequence of sets $S_{0} \subset S_{1} \subset \ldots$ so that $S_{j} \in \mathcal{Q}$ for all $j\geq 0$, and  for any compact set $E \in \mathbb{R}^{n}$ there exists $S_{\ell}$ such that $E \subseteq S_{\ell}$. We apply the lemma to the sets $S_{j}$:
\begin{align}\label{predel}
\int_{S_{j}}(M_{\mathcal{Q}}f)^{p} \geq |S_{j}|z_{S_{j}}^{p}(f)+\frac{\lambda^{p}-1}{\lambda^{p}-\lambda}\left( \int_{S_{j}} f^{p}- |S_{j}|x_{S_{j}}(f) z_{S_{j}}^{p-1}(f) \right)_{+}.
\end{align}
As $j \to \infty$, the left hand side of (\ref{predel}) tends to $\| M_{\mathcal{Q}} f\|_{L^{p}}^{p}$. On the right hand side of (\ref{predel}) we have $|S_{j}|z_{S_{j}}^{p}(f) \to 0$, $|S_{j}| x_{S_{j}}(f)z_{S_{j}}^{p-1} \to 0$, and $\int_{S_{j}}f^{p} \to \|f\|_{L^{p}}^{p}$ and the Theorem follows. Now we return to the proof of Lemma~\ref{ml}.

\begin{proof}
For $0 \leq x \leq z, \lambda >1$ and $p>1$ we define the Bellman function as follows
\begin{align*}
B(x,y,z; \lambda) = z^{p} + \frac{\lambda^{p}-1}{\lambda^{p}-\lambda}(y-xz^{p-1})_{+}.
\end{align*}
 Fix some locally integrable $f \geq 0$ such that $\int f \neq 0$. Pick any $Q \in \mathcal{Q}$, $|Q|>0$. 
By $\lambda$-density we have  $Q = \cup_{P \in \mathrm{ch}(Q)}P$ and $\av{f}{P}\leq \lambda \av{f}{Q}$. We show the following {\em main inequality}:
\begin{align}\label{mi}
B(x_{Q}, y_{Q}, z_{Q}; \lambda) \leq \sum_{P\in \mathrm{ch}(Q)} \frac{|P|}{|Q|} B(x_{P}, y_{P}, z_{P}; \lambda).
\end{align}

First notice that $z_{P} \geq \max\{x_{P}, z_{Q} \}$ for any $P\in \mathrm{ch}(Q)$. Since $B$ is increasing in $z$ in its domain, it is enough to prove (\ref{mi}) when the quantities $z_{P}$ are replaced by $\max\{ x_{P}, z_{Q}\}$. If $y_{Q}-x_{Q}z_{Q}^{p-1}\leq 0$, then the inequality is obvious because $B(x_{P},y_{P}, \max\{x_{P},z_{Q}\} ;\lambda)\geq z_{Q}^{p}$. So we assume $y_{Q}-x_{Q}z_{Q}^{p-1}>0$. Then we will prove the stronger inequality where all expressions of the form $(\,\cdots)_{+}$ on the right hand side of (\ref{mi}) are replaced by the lower bounds $(\,\cdots)$. Since $y_{Q} = \sum_{P} \frac{|P|}{|Q|}y_{P}$ and $x_{Q}=\sum_{P} \frac{|P|}{|Q|}x_{P}$, the inequality we wish to prove  takes the following form after dividing by $z_{Q}^{p}$
\begin{align*}
1  - \frac{\lambda^{p}-1}{\lambda^{p}-\lambda}\sum_{P \in \mathrm{ch}(Q)} \frac{|P|}{|Q|}\frac{x_{P}}{z_{Q}} \leq \sum_{P \in \mathrm{ch}(Q)} \frac{|P|}{|Q|}\left[\max\left\{\frac{x_{P}}{z_{Q}} ,1\right\}^{p} - \frac{\lambda^{p}-1}{\lambda^{p}-\lambda}\frac{x_{P}}{z_{Q}}\max\left\{ \frac{x_{P}}{z_{Q}},1\right\}^{p-1} \right].
\end{align*}
This can be rewritten as follows

\begin{align*}
\frac{\lambda^{p}-1}{\lambda^{p}-\lambda}\left[ \sum_{P\in \mathrm{ch}(Q)}\frac{|P|}{|Q|} \frac{x_{P}}{z_{Q}}\left( \max\left\{ \frac{x_{P}}{z_{Q}},1\right\}^{p-1}-1\right)\right]\leq \sum_{P\in \mathrm{ch}(Q)}\frac{|P|}{|Q|}\max \left\{ \frac{x_{P}}{z_{Q}},1\right\}^{p}-1.
\end{align*}
Replacing $\frac{x_{P}}{z_{Q}}$ by $\max\left\{\frac{x_{P}}{z_{Q}},1\right\}$ we see that it is enough to prove the following stronger (in fact equivalent) inequality 
\begin{align}\label{stronger}
\frac{\lambda^{p}-1}{\lambda^{p}-\lambda}\left[ \sum_{P\in \mathrm{ch}(Q)}\frac{|P|}{|Q|}\left( \max\left\{ \frac{x_{P}}{z_{Q}},1\right\}^{p}-\max\left\{ \frac{x_{P}}{z_{Q}},1\right\}\right)\right]\leq \sum_{P\in \mathrm{ch}(Q)}\frac{|P|}{|Q|}\max \left\{ \frac{x_{P}}{z_{Q}},1\right\}^{p}-1.
\end{align}
Now notice that if $1\leq s \leq \lambda$ then
\begin{align}\label{convexity}
\frac{\lambda^{p}-1}{\lambda^{p}-\lambda}(s^{p}-s)\leq s^{p}-1.
\end{align}
This is a consequence of the fact that the function  $s \mapsto s^{p}$ is  convex and its graph between $s=1$ and $s=\lambda$  lies below  corresponding chord. Since $1\leq  \max\left\{ \frac{x_{P}}{z_{Q}},1\right\} \leq \lambda$ the inequality (\ref{stronger}) follows by averaging (\ref{convexity}) at the points $\max\left\{\frac{x_{P}}{z_{Q}},1\right\}$ with the weights $\frac{|P|}{|Q|}$.  The main inequality (\ref{mi}) is proved.

We start with $S$ and iterate the main inequality $m$ times. We obtain 
\begin{align*}
B(x_{S}(f), y_{S}(f), z_{S}(f); \lambda)  \leq  \sum_{P \in \mathcal{F}_{m}} \frac{|P|}{|S|} B(x_{P}(f), y_{P}(f), z_{P}(f);\lambda).
\end{align*} 

Recall that by the definition of  $\lambda$-density  we have
\begin{align*}
\lim_{m \to \infty} \sup_{P \in \mathcal{F}_{m}}\mathrm{diam}(P)=0. 
\end{align*}

The family $\mathcal{Q}$ consists of  convex sets so the Lebesgue differentiation theorem implies  that in the limit we will obtain the desired result:
\begin{align*}
B(x_{S}(f), y_{S}(f), z_{S}(f); \lambda) \leq \av{B(f,f^{p}, M_{\mathcal{Q}}(f))}{S} = \av{(M_{\mathcal{Q}}(f)))^{p}}{S}.
\end{align*}
\end{proof}

\subsection*{Remark 1}\label{sun}
We would like to mention that in one dimensional case there is a simple proof to obtain the constant $(\frac{p}{p-1})^{1/p}$ as the lower bound for the uncentered maximal function operator defined over the intervals. The argument is due the Lerner~\cite{ler} (see also \cite{gr}). Indeed,  for $f \geq 0$ we consider $(M_{R}f)(x) = \sup_{b>x} \frac{1}{b-x} \int_{x}^{b} f(t) dt$. Then notice that 
\begin{align}\label{grafakos}
t | \{x : (M_{R} f)(x)>t \}| = \int_{\{x\,  :\,  (M_{R} f)(x)>t \}}f(s)ds
\end{align} 
(for example, see Lemma~1 in \cite{gr}). Now multiplying  both sides of  (\ref{grafakos}) by $t^{p-2}$ and integrating from $0$ to $\infty$  with respect to $dt$ we obtain $\frac{\| M_{R}f\|_{p}^{p}}{p} = \frac{1}{p-1}\int_{\mathbb{R}}f (M_{R}f)^{p}$. Therefore $\frac{\| Mf \|^{p}_{p}}{p} \geq \frac{1}{p-1}\int_{\mathbb{R}}f^{p}$ and the desired estimate follows. Unfortunately it is unclear how to use this argument to obtain the lower bounds for the maximal function operator defined  over the  $\lambda$-dense family in $\mathbb{R}^{n}$, $n \geq 2$.
\bigskip 
\subsection*{Remark 2}
 We  omit  the explanation of why we consider this special function $B(x,y,z; \lambda)$ because it is not necessary for the formal proof. However, the reason lies in the geometry of the solution of the relevant homogeneous Monge--Amp\`ere equation. The function appeared for the first time in ~\cite{MNS1}. It can be derived by  using the methods recently developed  in works of N.~Osipov, D.~Stolyarov, V.~Vasyunin, P.~Zatitskiy and P.~Ivanisvili (see e.g. \cite{iosvz1, iosvz3, pi1}). 
 %We consider the following extremal problem 
 %\begin{align*}
%B_{p}(x,y,z) \ddf \inf_{f \geq 0, \; f \in L^{p}(\mathbb{R})}\{ \av{(Mf)^{p}}{I} :\; (\av{f}{I}, \av{f^{p}}{I}, \sup_{R \supset I}\av{f}{R})=(x,y,z) \}. 
%\end{align*}
 
 %Using the homogeneity property of $B_{p}$ (namely that $B_{p}(x,y,z)=z^{p}B_{p}(x/z, y/z^{p},1)$), we fix $z=1$ and construct the candidate for the  convex function $B_{p}(x,y,1)$ in the domain $0<x \leq 1$ and $x^{p}\leq y$ so that it satisfies the homogeneous Monge--Amp\`ere equation, the boundary condition $B_{p}(x,x^{p},1)=1$ and the condition $\frac{\partial }{\partial z}B(x,y,z) \geq 0$. Then we try to choose the maximal solution  and in this way we find our candidate $B(x,y,z; 1)$. 
 
 %We should also mention that the idea of using the Bellman function approach for the dyadic maximal functions  was  introduced by  F.~Nazarov, S.~Treil and A.~Volberg (see \cite{NT, NTV1, Vo1}). Later this idea was  developed in \cite{mela,ssv, MNS1} in order to find the exact Bellman function for the dyadic maximal operator. 

\section*{Appendices}
\subsection{Besicovich covering lemma}
\begin{lemma}\label{bcl}
Let $\mathcal{K}_{1} \subset \mathcal{K}$, and let $E$ be the set of centers of the sets in $\mathcal{K}_{1}$. Assume that $E$ is a bounded subset of $\mathbb{R}^{n}$, and $\sup_{S \in \mathcal{K}_{1}} \mathrm{diam}(S)<\infty$. Then there exists a constant $B(n)>0$  depending  only on the dimension $n$,  and at most countable collection of sets $K_{j} \in \mathcal{K}_{1}$ such that $E \subset \cup K_{j}$ and each point of $x\in \mathbb{R}^{n}$ is covered by at most  $B(n)$ sets from the family $\{ K_{j}\}$.  
\end{lemma}

Notice that the above formulation of the Besicovitch covering lemma with origin symmetric convex sets is equivalent to the more commonly found formulation concerning a collection of balls, only with the usual Euclidean norm replaced by some other norm in $\R^n$.  A proof of the covering lemma in a general finite dimensional normed space can be found (with a bit of work) in \cite{B,M,F}.  For a simple proof, the reader can see F\"{u}redi and Loeb \cite{FL}, where it is moreover shown that $(1.001)^n\leq B(n)\leq 5^n$ for any choice of origin symmetric convex body.\\

We shall use the Besicovitch covering lemma to prove the general form of Stein's inequality.

\begin{lemma}\label{stein}There is a constant $C=C(n)>0$ such that if $K_0\in \mathcal{K}$ and $f\geq 0$ satisfies $\av{f}{K_0}=1$, then
$$\int_{K_0} f \ln (\max\{1,f\})dx \leq C\int_{K_0} M_0(f)dx.
$$
\end{lemma}

Notice that the lemma is proved for the centered maximal function  (i.e., the smallest maximal function).

\begin{proof}  Without loss of generality we may assume that $f$ is supported in $K_0$.  Let $t>1$.  For each $x\in \{f>t\}$, choose some set $K_x\in \mathcal{K}$ centered at $x$ that satisfies
$\av{f}{K_x}=t.$ Notice that for each $t>0$ the diameter of such sets $K_{x}$ are uniformly bounded.
Now use the Besicovitch covering lemma to extract a sequence $K_j$ from the collection $\{K_x: x\in \{f>t\}\}$ such that 

$\bullet$  $\bigcup_j K_j\supset \{f>t\}$, and

$\bullet$ the sets $K_j$ have bounded overlap (with overlap number $B(n)\leq 5^n$).  

For each $x\in K_j$, we have that the set in $\mathcal{K}$ given by the concentric double of $K_j$ shifted to be centred at $x$ contains the set $K_j$.  From this we deduce that
$M_0(f)\geq \tfrac{t}{2^n}$ on $\bigcup_j K_j$. 
Finally, notice that $|K_{j}|\leq |K_{0}|$ (because $t> 1=\av{f}{K_{0}}$). Since each set $K_j$ is centrally symmetric and  centered in $K_0$, it  is a subset of $2K_0$  and  $|K_j\cap K_0|\geq c(n)|K_j|$ for a constant $c(n)$ depending only on $n$.  Therefore,
$$|\{x\in K_0: M_0(f)(x)>2^{-n}t\}|\geq  \frac{1}{B(n)}\sum_j |K_j\cap K_{0}| =\frac{c(n)}{t B(n) }  \sum_j \int_{K_j} fdx \geq \frac{c(n)}{t B(n)}\int_{\{f>t\}} fdx.
$$
Integrating this expression over $t>1$ we obtain the  desired estimate. 
\end{proof}

\subsection{$\lambda$-density of parallelepipeds}
\begin{lemma}\label{ldens}
The set of all parallelepipeds in $\mathbb{R}^{n}$ with nonzero volume, such that the sides are parallel to  the fixed $n$ linearly independent vectors, is $\lambda$-dense for any $\lambda>1$.
\end{lemma}
\begin{proof}
Clearly if the family is $\lambda_{1}$-dense then it is $\lambda_{2}$ dense for any $\lambda_{2}\geq \lambda_{1}$. Therefore we consider the case when $2 \geq \lambda >1$.
Let $P$ be a parallelepiped from the family.  Take any hyperplane which is parallel to a facet $L^{+}$ of the parallelepiped and call it $H$. The facet $L^{+}$ has an opposite facet $L^{-}$. We consider those $H$  that intersect the parallelepiped and divide it into  two parts which we denote by $P_{-}^{H}$ and $P_{+}^{H}$ correspondingly ($P_{+}^{H}$ contains the facet $L_{+}$). First choose $H$ so that $|P_{-}^{H}|=|P_{+}^{H}|$. If $\av{f}{P^{H}_{-}}=\av{f}{P^{H}_{+}}$ then $\max\{ \av{f}{P^{H}_{\pm}}\}\leq \lambda \av{f}{P}$, and we say that the partition $P=P^{H}_{-}\cup P^{H}_{+}$ is a {\em good partition}. Otherwise, consider the case  $\av{f}{P^{H}_{-}} < \av{f}{P^{H}_{+}}$ (the case of the opposite inequality is similar). Then we start moving the hyperplane $H$ closer to the facet $L^{-}$. Let $H^{*}$ be the  hyperplane parallel to $L^{+}$  such that $\lambda  = \frac{|P|}{|P^{H^{*}}_{+}|}$. 
We move $H$ toward  $H^{*}$ while $\av{f}{P^{H}_{-}} < \av{f}{P^{H}_{+}}$. If at some position of $H$  the equality happens then we stop and say that the hyperplane $H$ in that position gives a {\em good partition} of $P$, $P=P_{-}^{H}\cup P_{+}^{H}$. If the equality never happens then we just choose $H=H^{*}$. Notice that in this case we have
\begin{align*}
\lambda \av{f}{P} = \lambda\left( \frac{|P^{H^{*}}_{-}|}{|P|}\av{f}{P^{H^{*}}_{-}}+\frac{|P^{H^{*}}_{+}|}{|P|}\av{f}{P^{H^{*}}_{+}}\right)\geq \lambda \frac{|P^{H^{*}}_{+}|}{|P|}\av{f}{P^{H^{*}}_{+}} = \max\{\av{f}{P^{H}_{\pm}} \}.
\end{align*}
The last equality follows from the fact that all the time we had $\av{f}{P^{H}_{-}} < \av{f}{P^{H}_{+}}$.

The advantage of this partition is that  the hyperplane $H$ which splits the parallelepiped $P$  into  two parallelepipeds $P=P_{-}\cup P_{+}$ satisfies the following properties
\begin{itemize}
\item[1.] We have $\max\{ \av{f}{P_{+}}, \av{f}{P_{-}}\} \leq \lambda \av{f}{P}$.
\item[2.] None of the parts $P_{-}$ and $P_{+}$ is too large,  i.e., $\max \left\{\frac{|P_{-}|}{|P|}, \frac{|P_{+}|}{|P|}\right\} \leq \frac{1}{\lambda}$
\end{itemize}

If we iterate the partition then it  almost gives us $\lambda$-density, except it might happen that the diameters of the smaller parallelepipeds will not tend to zero.

 In order to chop the parallelepiped so that the diameters of $P_{\pm}$ tend to zero uniformly in the process of iteration, sometimes we need to change the direction of the hyperplane $H$ (it should be parallel to different facets of $P$).  On each step, we consider the largest side (face of dimension 1) of the parallelepiped $P$ (if they are several, we pick any of them). We choose the direction of the hyperplane $H$ so that it is parallel to the facets transversal to the largest side. Since on each step we are cutting the largest side of the parallelepiped with the ratio separated uniformly from zero,  it is clear that if we set $\mathcal{F}_{0}=\{ P\}$, $\mathcal{F}_{1}=\{P_{-}, P_{+}\}$, $\mathcal{F}_{2}=\{P_{- -}, P_{- +} , P_{+-}, P_{++}\}$ etc.,  we will have
$$
\lim_{n \to \infty}\sup_{P \in \mathcal{F}_{n}} \mathrm{diam}(P)=0.
$$
\end{proof}

\end{document}